\documentclass[12pt,a4paper]{amsart}
\usepackage{mathrsfs}
\usepackage{amssymb,amsmath,amsthm,color}
\usepackage{graphicx,mcite}
\usepackage{cite}
\usepackage{hyperref}
\usepackage{url}
\usepackage{setspace}
\usepackage{enumerate}
\newtheorem{theorem}{Theorem}[section]
\newtheorem{lemma}[theorem]{Lemma}
\newtheorem{proof of lemma}[theorem]{Proof of Lemma}
\newtheorem{proposition}[theorem]{Proposition}

\theoremstyle{definition}

\newtheorem{remark}[theorem]{Remark}

\numberwithin{equation}{section}

\begin{document}

\title[uniqueness of the Fourier transform]
{Uniqueness of the Fourier transform on the Euclidean motion group}

\author{A. Chattopadhyay, D.K. Giri and R.K. Srivastava}

\address{Department of Mathematics, Indian Institute of Technology, Guwahati, India 781039.}
\email{{arupchatt@iitg.ac.in, deb.giri@iitg.ac.in, rksri@iitg.ac.in}}


\date{\today}

\dedicatory{Dedicated to Prof. Gadadhar Misra on the occasion of his 60th birthday.}


\begin{abstract}
In this article, we prove that if the Fourier transform of a certain integrable
function on the Euclidean motion group is of finite rank, then the function has
to vanish identically. Further, we explore a new variance of the uncertainty
principle, the Heisenberg uniqueness pairs on the Euclidean motion group as
well as on the product group $\mathbb R^n\times K,$ where $K$ is a compact group.
\end{abstract}

\maketitle

\section{Introduction}\label{section1}
In a fundamental article, M. Benedicks \cite{B} had generalized the Euclidean Paley-Wiener
theorem to the class of integrable functions. That is, support of an integrable function
$f$ on $\mathbb R^n$ and its Fourier transform $\hat f$ both cannot be of finite measure
simultaneously.

\smallskip

Thereafter, a series of analogous results to the Benedicks theorem and related problems
had been explored in various set ups, including the Heisenberg group and the Euclidean
motion groups (see \cite{NR,PS,PSi,PT,R,Sa,SST}). In particular, Narayanan and Ratnakumar
\cite{NR} had worked out an analogous result to the Benedicks theorem for the partial
compactly supported functions on the Heisenberg group in terms of the finite rank of
Fourier transform of the function. Further, in a recent article\cite{V},  Vemuri has
relaxed the compact support condition on the function by the finite support. Since the
group Fourier transform on the Heisenberg group is operator valued, the latter result
seems close to the classical Benedicks theorem. However, it would be a good question to
consider the case when the spectrum of the Fourier transform of an integrable function
would be supported on a thin uncountable set.

\smallskip

In the path-breaking article \cite{HR}, a major variation of the uncertainty principle
has been observed by Hedenmalm and Montes-Rodr\'iguez, in terms of the measures supported
on the curves. If the Fourier transform of a finitely supported Borel measure vanishes
on a thin set, then the measure can be determined.

\smallskip

In the article \cite{HR}, Hedenmalm and Montes-Rodr\'iguez have shown that the pair
(hyperbola, some discrete set) is a Heisenberg uniqueness pair. As a dual problem,
a weak$^\ast$ dense subspace of $L^{\infty}(\mathbb R)$  has been constructed to
solve the Klein-Gordon equation. Further, in the same article, a complete
characterization of the Heisenberg uniqueness pairs corresponding to any
two parallel lines has been established.
\smallskip

Let $\Gamma$ be a finite disjoint union of smooth
curves in $\mathbb R^2.$ Let $X(\Gamma)$ be the space of all finite
complex-valued Borel measure $\mu$ in $\mathbb R^2$ which is supported
on $\Gamma$ and absolutely continuous with respect to the arc length
 measure on $\Gamma$. The Fourier transform of $\mu$ is defined by
\[\hat\mu{(\xi,\eta)}=\int_\Gamma e^{-i\pi(x\cdot\xi+y\cdot\eta)}d\mu(x,y),\]
where $(\xi,\eta)\in\mathbb R^2.$ Let $\Lambda$ be a set in $\mathbb R^2.$
The pair $\left(\Gamma, \Lambda\right)$ is called a Heisenberg uniqueness
pair for $X(\Gamma)$ if any $\mu\in X(\Gamma)$ satisfies $\hat\mu\vert_\Lambda=0,$
implies $\mu$ is identically zero. For more details, we refer the article \cite{HR}.

\smallskip

Now, we state the main result on the Heisenberg uniqueness pairs due to
Hedenmalm and Montes-Rodr\'iguez \cite{HR}.

\begin{theorem}{\em\cite{HR}\label{th18}}
Let $\Gamma$ be the hyperbola $x_{1}x_{2}=1$ and $\Lambda_{\alpha,\beta}$ a lattice-cross
defined by
\[\Lambda_{\alpha,\beta}=\left(\alpha\mathbb Z\times\{0\}\right)\cup\left(\{0\}\times\beta\mathbb Z\right),\]
where $\alpha, \beta$ are positive reals. Then $\left(\Gamma,\Lambda_{\alpha,\beta}\right)$ is a Heisenberg
uniqueness pair if and only if $\alpha\beta\leq1$.
\end{theorem}

Further, the questions pertaining to Heisenberg uniqueness pair have been
studied in the plane as well as in the Euclidean spaces by several authors.
We skip writing more histories and details about the Heisenberg uniqueness
pairs, however, we would like to refer \cite{Ba,GJ,GS,GS1,JK,L,S1,S2,Sri,V1, V2}.

\smallskip

The question of Heisenberg uniqueness pairs in the higher dimension has been
first taken up by Gonzalez Vieli \cite{V1}. We worked out an analogous
result to Theorem \ref{th14} for the Euclidean motion groups.

\begin{theorem}{\em\cite{V1}}\label{th14}
Let $\Gamma=S^{n-1}$ be the unit sphere in $\mathbb R^n$ and $\Lambda$
be a sphere of radius $r$. Then $\left(\Gamma,\Lambda\right)$ is a HUP
if and only if $J_{(n+2k-2)/2}(r)\neq0 $ for all $k\in\mathbb Z_+.$
\end{theorem}

In this article, we emphasize on an analogue of the Benedicks theorem
to the Euclidean motion group $M(n).$ We prove that if the Fourier
transform of certain integrable functions is of finite rank, then the
function has to vanish identically. Further, we explore the possibility
of the Heisenberg uniqueness pairs for the Fourier transform on $M(n)$
as well as on the product group $\mathbb R^n\times K.$

\section{Notation and preliminaries}\label{section2}
Euclidean motion group $G=M(n)$ is the group of isometries of $\mathbb R^n$
that leaves invariant the Laplacian. Since the action of the special
orthogonal group $K=SO(n)$ defines a group of automorphisms on $\mathbb R^n$
via $y\mapsto ky+x,$ where $x\in\mathbb R^n$ and $k\in K,$ the group $M(n)$
can be identified as the semidirect product of $\mathbb R^n$ and $K.$ Hence
the group law on $G$ can be expressed as
\[\left(x, s \right)\cdot\left(y,t\right)=\left(x+ sy, st\right).\]
Since a right $K$-invariant function on $G$ can be thought as
a function on $\mathbb R^n,$ we infer that the Haar measure on $G$ can
be written as $dg=dxdk,$ where $dx$ and $dk$ are the normalized Haar
measures on $\mathbb R^n$ and $K$ respectively.
\smallskip

Let $\mathbb R_+=(0,\infty)$ and $M=SO(n-1)$ be the subgroup of $K$ that
fixes the point $e_n=(0,\ldots,0,1).$ Let $\hat M$ be the unitary dual
group of $M.$ Given a unitary irreducible representation $\sigma\in\hat M$
realized on the Hilbert space $\mathcal H_\sigma$ of dimension $d_\sigma,$
we consider the space $L^2(K,\mathbb C^{d_\sigma\times d_\sigma})$ consisting
of $d_\sigma\times d_\sigma$ complex matrices valued functions $\varphi$ on $K$
such that $\varphi(uk)=\sigma(u)\varphi(k),$ where $u\in M,~k\in K$  and satisfying
\[\int_K\|\varphi(k)\|^2dk=\int_K\text{tr}(\varphi(k)^\ast\varphi(k))dk.\]
It is easy to see that $L^2(K,\mathbb C^{d_\sigma\times d_\sigma})$ is a Hilbert space under
the inner product \[\langle\varphi,\psi\rangle=\int_K\text{tr}(\varphi(k)\psi(k)^\ast)dk.\]
For each $(a,\sigma)\in\mathbb R_+\times\hat M,$ defines a unitary representation $\pi_{a,\sigma}$
of $G$ by
\begin{equation}\label{exp50}
\pi_{a,\sigma}(g)(\varphi)(k)=e^{-ia\left\langle x,k\cdot e_n\right\rangle}\varphi(s^{-1}k),
\end{equation}
where $\varphi\in L^2\left(K,\mathbb C^{d_\sigma\times d_\sigma}\right).$
Let $\varphi=(\varphi_1,\ldots,\varphi_{d_\sigma}),$ where $\varphi_j$ are
the column vectors of $\varphi.$ Then $\varphi_j(uk)=\sigma(u)\varphi_j(k).$
Now, consider the space
\[H\left(K,\mathbb C^{d_\sigma}\right)=
\left\{\varphi:K\rightarrow\mathbb C^{d_\sigma}, \int_K|\varphi(k)|^2dk<\infty,\varphi(uk)=\sigma(u)\varphi(k),u\in M\right\}.\]
It is obvious that $L^2(K,\mathbb C^{d_\sigma\times d_\sigma})$ is the direct sum of $d_\sigma$
copies of the Hilbert space $H(K,\mathbb C^{d_\sigma})$ equipped with the
inner product \[\langle\varphi,\psi\rangle=d_\sigma\int_K\left(\varphi(k),\psi(k)\right) dk.\]
Now, it can be shown that an infinite dimensional unitary irreducible
representation of $G$ is the restriction of $\pi_{a,\sigma}$ to $H(K,\mathbb C^{d_\sigma}).$
In other words, each of $(a,\sigma)\in\mathbb R_+\times\hat M$ defines a principal
series representation $\pi_{a, \,\sigma}$ of $G$ via (\ref{exp50}).

\smallskip

In addition to the principal series representations, there are finite-dimensional
unitary irreducible representations of $G$ which can be identified with $\hat K,$
though these unitary representations do not take part in the Plancherel formula.
For more details, we refer to Kumahara \cite{Ku0} and Sugiura \cite{Su}.

\smallskip

Now, we define the group Fourier transform of a function $f\in L^1(G)$ by
\[\hat f(a,\sigma)=\int_G f(g)\pi_{a,\,\sigma}(g^{-1})dg\]
and \[\hat f(\delta)=\int_G f(x, k)\delta(k^{-1})dxdk,\]
where $\delta\in\hat K.$ Further, the operator $\hat f(a,\sigma)$ can
be explicitly written as
\begin{eqnarray}\label{exp41}
(\hat f(a,\sigma)\varphi)(k)&=&\int_{\mathbb R^n}\int_Kf(x,s)e^{-i\left\langle x,\,ak\cdot e_n\right\rangle}\varphi(s^{-1}k)dxds,\nonumber\\
&=&\int_{K}\mathcal F_1f(ak\cdot e_n,s)\varphi(s^{-1}k)ds,
\end{eqnarray}
where $\mathcal F_1$ stands for the usual Fourier transform in the
first variable  and $\varphi\in H(K, \mathbb C^{d_\sigma}).$ For
more details, we refer to \cite{F,GK,Ku}.

\smallskip

Now, if $f\in L^1\cap L^2(G),$ then the operator $\hat f(a,\sigma)$ will be
a Hilbert-Schmidt operator on $H(K,\mathbb C^{d_\sigma}).$ Since the Plancherel
measure $\mu$ on $\hat G$ can be expressed as $d\mu=c_na^{n-1} da,$
where $c_n$ depends only on $n,$ the corresponding Plancheral formula is given by
\begin{equation}\label{exp48}
\int_0^\infty\left(\sum\limits_{\sigma\in\hat M}d_\sigma\|\hat f(a,\sigma)\|_{HS}^2\right)d\mu(a)=\|f\|_2^2.
\end{equation}

\smallskip

We would like to mention that our main result Theorem \ref{th8} has a close
relation with the following Wiener's theorem on motion group due to R. Gangolli
\cite{G}. For a function $f$ on $G,$ defining the two-sided translate by
$^{g_1}f^{g_2}(g)=f(g_1gg_2^{-1}),$ where  $g_j\in G; j=1,2.$

\begin{theorem}{\em\cite{G}}\label{th11}
Let $f\in L^1(G)$ and $S=\text{ span  }\{^{g_1}f^{g_2}: g_j\in G; j=1,2\}.$
Then the space $S$ is dense in $L^1(G)$ if and only if $\hat f(a,\sigma)\neq0$
and $\hat f(\delta)\neq0$ for all  $(a,\sigma)\in\mathbb R_+\times\hat M$ and
$\delta\in\hat K.$
\end{theorem}

A close observation of Theorem \ref{th11} shows that if $\hat f(a,\sigma)$ is
a finite rank operator, then $\overline S$ can be a proper subspace of $L^1(G).$
Hence, it opens a window to look at the determining properties of $\hat f.$

\smallskip

Next, we recall certain facts about the spherical harmonics. Let $\hat{K}_M$
denote the set of all equivalence classes of irreducible unitary representations
of $K$ which have a nonzero $M$-fixed vector. It is well known that each
representation in $\hat{K}_M$ has, in fact, a unique nonzero $M$-fixed vector,
up to a scalar multiple.

\smallskip
For a $\delta \in \hat{K}_M,$ which is realized on $V_{\delta},$ let $\{e_1,\ldots,e_{d_\delta}\}$
be an orthonormal basis of $V_{\delta},$ with $e_1$ as the $M$-fixed vector.
Let $\varphi_{ij}^\delta(k)=\langle e_i,\delta(k)e_j \rangle,$ $k\in K.$
Then by  the Peter-Weyl  theorem, it follows that
$\{\sqrt{d_\delta}\varphi_{1j}^\delta:1\leq j\leq d_\delta,\delta\in\hat{K}_M\}$
is an orthonormal basis of $L^2(K/M).$

\smallskip

We would further need a concrete realization of the representations in $\hat{K}_M,$
which can be done in the following way.

\smallskip
For $l\in \mathbb Z_+,$ denote the set of all non-negative integers, let $P_l$
denote the space of all homogeneous polynomials $P$ in $n$ variables of degree $l.$

\bigskip

Let $H_l=\{P\in P_l: \Delta P=0\},$ where $\Delta$ is the standard Laplacian
on $\mathbb R^n.$ The elements of $H_l$ are called solid spherical harmonics
of degree $l.$ It is easy to see that the natural action of $K$ leaves the
space $H_l$ invariant. In fact, the corresponding unitary representation
$\delta_l$ is in $\hat{K}_M.$ Moreover, $\hat{K}_M$ can be identified, up
to unitary equivalence, with the collection $\{\delta_l: l\in\mathbb Z_+\}.$

\smallskip

Define $Y_{lj}(\omega)=\sqrt{d_l}\varphi_{1j}^{\delta_l}(k),$ where
$\omega=k\cdot e_n\in S^{n-1},$ $k\in K$ and $d_l$ is the dimension of
$H_l.$ Then the set $\widetilde H_l=\left\{Y_{lj}: 1\leq j\leq d_l \text{ and }l\in\mathbb Z_+\right\}$
forms an orthonormal basis for $L^2(S^{n-1}).$ Thus, a suitable function
$g$ on $S^{n-1}$ can be expanded as
\begin{equation}\label{exp10}
g(\omega)=\sum_{l=0}^{\infty}\sum_{j=1}^{d_l}~a_{lj}Y_{lj}(\omega).
\end{equation}
These spherical functions $Y_{lj}$ are called the spherical harmonics on the unit
sphere $S^{n-1}.$ For more details, see \cite{T2}, p. 11.

\smallskip

Next, we consider an orthogonality relation among the matrix coefficients of
the irreducible unitary representations in $\hat K_M.$
\begin{lemma}\label{lemma4}
For $\delta\in\widehat K,$ define $\phi^\delta_{ij}(k)=\left\langle e_i,\delta(k)e_j\right\rangle.$
Then for $\delta_1,\delta_2\in\hat K,$ there exists $\alpha\in\mathbb Z_+$
such that
\begin{equation}\label{exp85}
\int_M\phi^{\delta_1}_{ij}(km)\overline{\phi^{\delta_2}_{lu}(km)}dm=\sum_{v=0}^\alpha c_vY_v(k\cdot e_n).
\end{equation}
\end{lemma}
\begin{proof}
Since we know that the matrix coefficients of $\delta\in\hat K_M$ satisfy the
functional relation
\begin{equation}\label{exp15}
\phi^\delta_{ij}(km)=\sum_{p=1}^{d_\delta}\phi^\delta_{ip}(k)\phi^\delta_{pj}(m)
\end{equation}
and $M$ is a compact subgroup of $K,$ it follows that each of $\delta\in\hat K$
will be the direct sum of finitely many irreducible unitary representations of
$M.$ Hence each of $\phi^\delta_{pj}$ satisfies
\[\phi^\delta_{pj}=\sum_{q=1}^{d_{\delta,\beta}}\phi^{\delta_q}_{pj},\]
where $\delta_q\in\hat M.$ Using the orthogonality of the coefficients $\phi^{\delta_q}_{pj}$'s
and the fact that the left-hand side of (\ref{exp85}) is $M$-invariant, we infer that
it is a finite sum of the product of some spherical harmonics. Further, a homogeneous
polynomial can be uniquely decomposed in terms of homogeneous harmonics polynomials,
it follows that (\ref{exp85}) holds.
\end{proof}

For a fixed $\xi\in S^{n-1},$ we define a linear functional on $H_l$ by
$\xi\mapsto Y_l(\xi).$ Then there exists a unique spherical harmonic, say
$Z_\xi^{(l)}\in H_l$ such that
\begin{equation}\label{exp1}
Y_l(\xi)=\int_{S^{n-1}}Z_\xi^{(l)}(\eta)Y_l(\eta)d\sigma(\eta).
\end{equation}
The spherical harmonic $Z_\xi^{(l)}$ is a $K$ bi-invariant real-valued function
which is constant on the geodesics those are orthogonal to the line joining the
origin and $\xi.$ The spherical harmonic $Z_\xi^{(l)}$ is called the zonal harmonic of the
space $\widetilde H_l$ for the above and various other peculiar reasons. For
more details, see \cite{SW}.

\smallskip

Since the zonal harmonic $Z_\xi^{(l)}(\eta)$ is $K$ bi-invariant, there
exists a reasonable function $F$ on $(-1,1)$ such that $Z_\xi^{(l)}(\eta)=F(\xi\cdot\eta).$
Hence, the extension of the formula (\ref{exp1}) is inevitable. For
$F\in L^1(-1,1),$ the Funk-Hecke identity is
\begin{equation}\label{exp3}
\int_{S^{n-1}}F(\xi\cdot\eta)Y_l(\eta)d\sigma(\eta)= c_lY_l(\xi),
\end{equation}
where the constant $c_l$ is given by
\[c_l=\alpha_l\int_{-1}^1F(t)G_l^{\frac{n-2}{2}}(t)(1-t^2)^{\frac{n-3}{2}}~dt\]
and $G_l^\beta$ stands for the Gegenbauer polynomial of order $\beta$ and degree $l.$

\smallskip

Further, using the Funk-Hecke identity, it can be shown that
\begin{equation}\label{exp16}
\int_{S^{n-1}} e^{-  i{x\cdot\eta}}Y_j(\eta)d\sigma(\eta)=  i^j~\frac{J_{j+(n-2)/2}(r)}{r^{(n-2)/2}}Y_j(\xi),
\end{equation}
whenever $Y_j\in\widetilde H_l.$ For a proof of the identity (\ref{exp16}), we refer \cite{AAR}, p. 464.

\smallskip

Let $f$ be a function in $L^1(S^{n-1}).$ For each $l\in\mathbb Z_+,$ we define
the $l^{th}$ spherical harmonic projection of the function $f$ by
\begin{equation}\label{exp11}
\Pi_lf(\xi)=\int_{S^{n-1}}Z_\xi^{(l)}(\eta)f(\eta)d\sigma(\eta).
\end{equation}
Then $\Pi_lf$ is a spherical harmonic of degree $k.$ Now, for $\delta>(n-2)/2,$
if we denote $A_l^p=\binom{p-l+\delta}{\delta}{\binom{p+\delta}{\delta}}^{-1},$
then the spherical harmonic expansion $\sum\limits_{l=0}^\infty\Pi_lf$
is $\delta$- Cesaro summable to $f.$ In other words,
\begin{equation}\label{exp2}
f=\lim\limits_{p\rightarrow\infty}\sum_{l=0}^pA_l^p\Pi_lf,
\end{equation}
where the limit on the right-hand side of (\ref{exp2}) exists in $L^1\left(S^{n-1}\right).$
For more details, we refer to \cite{K,So}.

\section{Uniqueness results for the Fourier transform on $G$}\label{section3}
In this section, we work out some of the uniqueness results for the Fourier transform
on the motion group $G$ as an analogue to the Benedick's theorem. We prove that
if the Fourier transform of the function $f\in L^1\cap L^2(G)$ is a finite rank
operator on $H(K,\mathbb C^{d_\sigma}),$ then the function has to vanish identically.

\smallskip

\begin{theorem}\label{th8}
Let $f\in L^1\cap L^2(G)$ be such that $\hat f(a,\sigma)$ is a finite rank
operator for each $a>0$ and some $\sigma\in\hat M.$ Then $f=0$ if and only if
$\hat f(\delta)\neq0$ except for finitely many $\delta\in\hat K.$
\end{theorem}

\begin{proof}
Let $h=f\ast f^\ast,$ where $f^*(g)=\overline{f(g^{-1})}.$ Then it can be shown
that $\hat h(a,\sigma)=\hat f(a,\sigma)^\ast\hat f(a,\sigma),$ (see \cite{Su}, p. 170).
Hence $\hat h(a,\sigma)$  is a positive, finite rank operator on $H(K,\mathbb C^{d_\sigma}).$
By the spectral theorem, it follows that
\begin{equation}\label{exp65}
\hat h(a,\sigma)\varphi=\sum_{j=1}^m\lambda_j\left\langle\varphi,\varphi_j\right\rangle\varphi_j,
\end{equation}
where the set $\{\varphi_j\in H(K,\mathbb C^{d_\sigma}):j=1,\ldots,m\}$ forms an
orthonormal basis for the range space of $\hat h(a,\sigma)$ which satisfies
$\hat h(a,\sigma)\varphi_j=\lambda_j\varphi_j$ with $\lambda_j\geq0.$
Let $\varphi_{j}=\left(\varphi_{j,1},\ldots,\varphi_{j,d_\sigma}\right).$
Then in view of (\ref{exp41}), we can express
\begin{equation}\label{exp81}
\int_{K}\mathcal F_1h(ak\cdot e_n,s)\varphi_{j,q}(s^{-1}k)ds=\lambda_j\varphi_{j,q}(k),
\end{equation}
where $q\in\{1,\ldots,d_\sigma\}.$ Since $h\in L^1(G),$ we can write the spherical harmonic
decomposition of $h$ in the first variable $x=|x|t,~t\in S^{n-1}$ as
\begin{equation}\label{exp84}
h(x,s)=\lim\limits_{p\rightarrow\infty}\sum_{l=0}^pA_l^ph_l(|x|,s)\Pi_lh(t,s),
\end{equation}
where the series on the right-hand side is $\delta$- Cesaro summable. Now, by the
Hecke-Bochner identity, we obtain
\begin{equation}\label{exp87}
\mathcal F_1\left(h_l(., s)\Pi_lh(.,s)\right)(a\omega)=i^{-l}a^{l}(\mathcal F_{n+2l}H_l)(a, s)\Pi_lh(t,s),
\end{equation}
where $\mathcal F_{n+2l}$ is the $(n+2l)$-dimensional Fourier transform of $H_l=\frac{h_l}{|x|^{l}}.$
In view of (\ref{exp84}) and (\ref{exp87}), we can rewrite (\ref{exp81}) as
\[\int_{K}\lim\limits_{p\rightarrow\infty}\sum_{l=0}^pA_l^p i^{-l}a^{l}(\mathcal F_{n+2l}H_l)(a,s)\Pi_lh(t,s)\varphi_{j,q}(s^{-1}k)ds=\lambda_j\varphi_{j,q}(k)\]
By using the orthogonality relations of the spherical harmonics, we infer that
\begin{equation}\label{exp73}
\int_{K}(\mathcal F_{n+2l}H_l)(a,s)\varphi_{j,q}(s^{-1}k)ds =\left\{
\begin{array}{ll}
\lambda_{j,q}\varphi_{j,q}(k) , & \hbox{ if }~ l=0 \\
0, & \hbox{ if }~ l\neq0.
\end{array}
\right.
\end{equation}
Let $G_a(s)=\mathcal F_nh_o(a,s).$ Then from (\ref{exp73}), it follows that
\begin{equation}\label{exp88}
G_a\ast\varphi_{j,q}=\lambda_j\varphi_{j,q}.
\end{equation}
For a function $\phi\in L^1(K)$ and $\delta\in\hat K,$ define $\phi^\delta=\phi\ast\chi_\delta,$
where $\chi_\delta$ is the character of the representation $\delta.$  Then $\phi^\delta$
is class function and hence
 \[\widehat{\phi^\delta}(\eta)=\frac{1}{d_\eta}\left\langle \phi^\delta, \chi_\eta\right\rangle I
=\frac{1}{d_\eta}\left\langle\phi,\chi_\delta\right\rangle\left\langle\chi_\delta, \chi_\eta\right\rangle I .\]
Thus, from (\ref{exp88}) we get

\begin{equation}\label{exp77}
\widehat{\varphi^\delta}_{j,q}(\delta)\left(\widehat{G_a^\delta}(\delta)-\lambda_j\right)=0.
\end{equation}
Then $\widehat{G_a^\delta}(\delta)=\lambda_j,$ for finitely many $\delta\in\hat K,$
otherwise, by the Riemann-Lebesgue lemma, it follows that $\lambda_j=0,$ whenever
$j=1,\ldots,m.$ Hence from (\ref{exp65}) we get $\hat h(a,\sigma)=0.$ In view of
(\ref{exp41}), we infer that $\mathcal F_1h(ak\cdot e_n,s)=0$ for almost all
$s,k\in K$ and all $a>0.$ This, in turn, by the uniqueness of the Fourier
transform $\mathcal F_1,$ it follows that $h=0.$ Since $h=f\ast f^\ast$
is continuous,  we can write $h(o)=||f||_2^2.$ Thus, $f=0.$

\smallskip

Now, we need to resolve the case when $\widehat{G_a^\delta}(\delta)=\lambda_j\neq0$
for the only finitely many $\delta\in\hat K.$ Thus, each of $\varphi_{j,q}$ is a
trigonometric polynomial in $L^2(K).$ By trigonometric polynomial we mean a finite
linear combination of matrix coefficients $\varphi_{ij}^\delta.$

\smallskip

Notice that, $\delta$ and $j$ are independent and the the fact that all of $\lambda_j$
cannot be zero simultaneously, it follows that $\lambda_j$'s are equal. Since
$\varphi_{j,q}$ is a trigonometric polynomial, $\pi_{a,\sigma}(g)\hat h(a,\sigma)$
will be a trace class operator on $H(K,\mathbb C^{d_\sigma}).$

\smallskip

Let $\psi^\delta_{lu}=\sqrt{d_\delta}\phi^{\delta}_{lu}$ denotes the matrix coefficients
of $\delta\in\hat K.$ Then by the Peter-Weyl
theorem, the set $\{\psi^{\delta}_{lu}:1\leq l, u\leq d_\delta, \delta\in\hat K\}$
forms an orthonormal basis for $L^2(K).$ Let $\{\tau_\nu: \nu\in\mathbb N\}$
be an orthonormal basis for $H(K,\mathbb C^{d_\sigma}).$ Then
\begin{eqnarray}\label{exp69}
\text{tr}\left(\pi_{a,\sigma}(x,s)\hat h(a,\sigma)\right)
&=&\sum_{\nu\in\mathbb N}\left\langle\pi_{a,\sigma}(x,s)\hat h(a,\sigma)\tau_\nu,\tau_\nu\right\rangle\nonumber\\
&=&\sum_{\nu\in\mathbb N}\sum_{j=1}^m\lambda\left\langle\tau_\nu,\varphi_j\right\rangle\left\langle\pi_{a,\sigma}(x,s)\varphi_j,\tau_\nu\right\rangle\nonumber\\
&=&\sum_{\nu\in\mathbb N}\sum_{j=1}^m\alpha_j\sum_{q=1}^{d_\sigma}\int_{K}e^{-i\left\langle x,ak\cdot e_n\right\rangle}\varphi_{j,q}(s^{-1}k)\overline{\tau_{\nu,q}(k)} dk,
\end{eqnarray}
where $\alpha_j=\lambda d_\sigma\left\langle\tau_\nu,\varphi_j\right\rangle.$
Since $\varphi_{j,q}$ is a trigonometric polynomial, there exists a finite
set $F_o$ in $\hat K$ such that
\begin{equation}\label{exp690}
\varphi_{j,q}=\sum_{\delta\in F_o}\sum_{l,u=1}^{d_\delta}c^\delta_{lu}\psi^\delta_{lu},
\end{equation}
where $c^\delta_{lu}$'s are constant. Now, in view of (\ref{exp15}), we can express
\[\psi^\delta_{lu}(s^{-1}k)=\sum_{p=1}^{d_\delta}\psi^\delta_{lp}(s^{-1})\psi^\delta_{pu}(k).\]
On the other hand, since $\tau_{\nu,q}\in L^2(K),$ by the Peter-Weyl theorem, we get
\begin{equation}
\tau_{\nu,q}=\sum_{\delta\in\hat K}\sum_{\xi,\eta=1}^{d_\delta}\kappa^\delta_{\xi\eta}\psi^\delta_{\xi\eta},\\
\end{equation}
where $\kappa^\delta_{\xi\eta}$'s are constants. In view of the fact that $e^{-i\left\langle x,ak\cdot e_n\right\rangle}$
is an $M$-invariant function, the followings identities hold.
\begin{eqnarray*}\label{exp82}
\int_Ke^{-i\left\langle x,ak\cdot e_n\right\rangle}\psi^\delta_{pu}(k)\overline{\psi^\delta_{\xi\eta}(k)}dk
&=&\int_M\int_Ke^{-i\left\langle x,a~km\cdot e_n\right\rangle}\psi^{\delta}_{pu}(k)\overline{\psi^\delta_{\xi\eta}(k)}dk dm\\
&=&\int_Ke^{-i\left\langle x,ak\cdot e_n\right\rangle}\int_M\psi^{\delta}_{pu}(km^{-1})\overline{\psi^\delta_{\xi\eta}(km^{-1})}dmdk.
\end{eqnarray*}
By using Lemma \ref{lemma4} and the Funk-Hecke identity, we infer that
\begin{eqnarray*}\label{exp83}
\int_{K}e^{-i\left\langle x,ak\cdot e_n\right\rangle}\psi^\delta_{pu}(k)\overline{\psi^\delta_{\xi\eta}(k)}dk
&=&\sum_{j=0}^\beta c_j\int_{K}e^{-i\left\langle x,ak\cdot e_n\right\rangle}Y_j(k\cdot e_n)dk\\
&=&\sum_{j=0}^\beta c_j'\phi_j(a|x|)Y_j(\omega)
\end{eqnarray*}
where $x=|x|\omega,~\omega\in S^{n-1}$ and $\phi_j(a|x|)=\dfrac{J_{j+(n-2)/2}(a|x|)}{(a|x|)^{(n-2)/2}}.$
Now, in view of the Fourier inversion formula (\cite{Su}, p. 175) for function in $L^2(M(2)),$
an inversion formula for the function $h\in L^2(M(n))$ can be deduced in a similar way and hence
we omit its proof here. Thus, the function $h$ can be recovered at $(x,s)$ by
\[h(x,s)=c_n\sum\limits_{\sigma\in\hat M}\int_0^\infty\text{tr}\left(\pi_{a,\sigma}(g)\hat h(a,\sigma)\right)a^{n-1}da,\]
Since $J_{j+(n-2)/2}(a|x|)  \approx(a|x|)^{-1/2}$ as $|x|\rightarrow\infty,$
(see \cite{Sz}, p. 15), it follows that $\phi_j\in L^p(\mathbb R^n)$ if and
only if $p>\frac{2n}{n-1}.$ This contradicts the hypothesis that $h\in L^1(G)$
and hence $h=0.$
\end{proof}

\begin{remark}
We would like to mention the necessity of the non-vanishing  of the Fourier
coefficients in Theorem \ref{th8}. Since $M(2)$ is the semidirect product of
$\mathbb R^2$ and $SO(2),$ each of $a\in\mathbb R_+$ defines a unitary irreducible
representation $\pi^a$ of $M(2)$  on $L^2([0,2\pi]).$ In other words, for
$(x,\theta)\in\mathbb R^2\times[0,2\pi],$ the action of $\pi^a$ can be realized by
\[(\pi^a(x,\theta)\varphi)(\omega)=e^{-i\left\langle x,ae^{i\omega}\right\rangle}\varphi(\omega-\theta),\]
where $\varphi\in L^2([0,2\pi]).$ For $g\in L^1(\mathbb R^2),$ consider
$f(x,\theta)=g(x)e^{in\theta}.$ Then
\[(\hat f(a)\varphi)(\omega)=\int_{\mathbb R^2}\int_{0}^{2\pi}
f(x,\theta)e^{-i\left\langle x,ae^{i\omega}\right\rangle}\varphi(\omega-\theta)dxd\theta=\hat g(ae^{i\omega})\varphi(\omega).\]
Hence we infer that, if the Fourier coefficients \[\delta_n(f)=\int_{M(2)} f(x,\theta)e^{in\theta}dxd\theta\neq0\]
for finitely many $n\in\mathbb Z,$ then $\hat f(a)$ is a finite rank operator,
however, $f$ need not be the zero function.
\end{remark}

Next, we prove a uniqueness result for the Fourier transform on $G$
which has a sharp contrast with the Benidicks theorem. That is, the group
Fourier transform of a non-zero bounded Borel measurable function in
$L^1(G)$ can not be compactly supported in $(0,\infty).$

\smallskip

In order to prove this result, we need the following result from \cite{BS}.
Let $\mathbb R^n_+=\left\{\left(x_1,\ldots, x_n\right)\in\mathbb R^n: x_j\geq0;~j=1,\ldots,n\right\}.$
The following result had appeared in the article \cite{BS} by Bagchi and Sitaram, p. 421,
as a part of the proof of Proposition $2.1.$

\begin{proposition}\label{prop1}{\em\cite{BS}}
Let $h$ be a non-zero bounded Borel measurable function which is supported on $\mathbb R^n_+.$
Then $\text{supp } \hat h=\mathbb R^n.$
\end{proposition}

\begin{theorem}\label{th64}
Let $f\in  L^1(G)$ be a bounded Borel measurable function and supported away from the origin
in the first variable. If $\hat f(., \sigma)$ is compactly supported in $\mathbb R_+,$ then
$f=0.$
\end{theorem}
\begin{proof}
Suppose $f$ is a radial function in the first variable. Then $\mathcal F_1f(.,s)$ will
be radial and hence
\begin{equation}\label{exp51}
(\hat f(a,\sigma)\varphi)(k) =\int_{K} \mathcal F_1f(a, s)\varphi(s^{-1}k)ds,
\end{equation}
where $\varphi\in H(K, \mathbb C^{d_\sigma}).$ Since $\hat f(., \sigma)$ is compactly
supported in $\mathbb R_+,$ it follows from (\ref{exp51}) that $\mathcal F_1f(.,s)$
is compactly supported in $\mathbb R_+,$ for almost all $s\in K.$ This, in turn,
contradicts Proposition \ref{prop1}. Thus, we conclude that $f=0.$

\smallskip

Since $f\in L^1(G),$ in view of (\ref{exp2}), we can write the spherical harmonic
decomposition of $f$ in the first variable $x=|x|t,~t\in S^{n-1}$ as
\begin{equation}\label{exp703}
f(x,s)=\lim\limits_{p\rightarrow\infty}\sum_{l=0}^pA_l^pf_l(|x|,s)\Pi_lf(t,s),
\end{equation}
where the series on the right-hand side is $\delta$-Cesaro summable. Now, an
application of the Hecke-Bochner identity to (\ref{exp51}) yields
\[(\hat f(a,\sigma)\varphi)(k)=
\int_{K}\lim\limits_{p\rightarrow\infty}\sum_{l=0}^pA_l^pi^{-l}a^{l}\mathcal F_{n+2l}H_l(a, s)\Pi_lf(t,s)\varphi(s^{-1}k)ds,\]
where $\mathcal F_{n+2l}$ is the $(n+2l)$-dimensional Fourier transform of
$H_l=\frac{f_l}{|x|^{l}}.$ Since $\hat f$ is compactly supported, it follows
that
\begin{equation}\label{exp4}
\lim\limits_{p\rightarrow\infty}\sum_{l=0}^pA_l^pi^{-l}a^{l}\mathcal F_{n+2l}H_l(a, s)\Pi_lf(t,s)=0.
\end{equation}
We know that the set $\{\Pi_lf(\cdot,s):~l\in\mathbb Z_+\}$ form an orthogonal set in
$L^2(S^{n-1}),$ from (\ref{exp4}), it is easy to see that
\[\mathcal F_{2+2l}H_l(a,s)\left\|\Pi_lf(.,s)\right\|_2^2=0.\]
Since $f$ is supported away from the origin in the first variable, $H_l$ must be a bounded
Borel measurable function. If $\mathcal F_{2+2l}H_l(a,s)=0,$ then by the  radial
case, we infer that $H_l=0.$ Otherwise, $\left\|\Pi_lf(.,s)\right\|_2=0.$
Thus, it follows from (\ref{exp703}) that $f=0.$
\end{proof}

Further, we prove that a radial function on $G$ can be determined by its
Fourier transform at a single point.
\begin{proposition}\label{prop4}
Let $f\in L^1(G)$ be a radial function in the first variable such that
$\text{sign}(J_0 f)\geq0.$ If $\hat f(a_o,\sigma)=0$ for some $a_o\in\mathbb R_+$ and
a fixed $\sigma\in\hat M,$ then $f=0.$
\end{proposition}
\begin{proof}
For $\varphi\in H(K, \mathbb C^{d_\sigma}),$ we have
\[(\hat f(a_o, \sigma)\varphi)(k)=\int_{K} \mathcal F_1f(a_o,s)\varphi(s^{-1}k)ds.\]
By the hypothesis, $\hat f(a_o, \sigma)\varphi=0,$ it follows that $\mathcal F_1f(a_o, .)=0.$
Hence
\begin{eqnarray*}\label{exp78}
\mathcal F_1f(a_o, s)&=&\int_0^\infty\int_{S^{n-1}}f(|t\omega|,s)e^{-ia_op\cdot t\omega}d\omega t^{n-1}dt\nonumber\\
&=&\int_0^\infty J_0(a_ot)f(t,s)t^{n-1}dt=0.
\end{eqnarray*}
Since $\text{sign}(J_0 f)\geq0$ and the Bessel function $J_0$ can vanish only at
the countably many points,  we conclude that $f=0.$
\end{proof}

\section{Some auxiliary results on compact group}\label{section4}
In this section, we observe some of the properties of a Weyl type transform on the
space $L^1(K)$ as analogous to the Weyl transform on the Heisenberg group, (see \cite{T}).
We use it to work out some uniqueness results for the Fourier transform on the motion
group $G.$

\smallskip

Let $K$ be a compact group. For a function $g\in L^1(K),$ we define an operator $W$
on $H(K, \mathbb C^{d_\sigma})$ by \[W(g)=\int_{K} g(t)\pi(t)dt,\] where $\pi$ is
the left regular representation of $K.$ Then $W(g)$ maps $H(K, \mathbb C^{d_\sigma})$
into $H(K, \mathbb C^{d_\sigma}).$ Now, we derive the Plancheral formula and the
Fourier inversion formula for the transform $W.$

\smallskip

\noindent{\bf{Plancherel formula.}}
For $g\in L^2(K)$ and $\varphi\in H(K, \mathbb C^{d_\sigma}),$ we have
\begin{eqnarray*}
(W(g)\varphi)(k)&=&\int_{K} g(t)(\pi(t)\varphi)(k)dt=\int_{K} g(t)\varphi(t^{-1}k)dt\\
&=&\int_{K} g(s^{-1}k)\varphi(s)ds.
\end{eqnarray*}
Write $\mathcal K_g(s,k)=g(s^{-1}k).$ Then $W(g)$ is an integral operator with the
kernel $\mathcal K_g\in L^2(K\times K).$ Hence $W(g)$ is a Hilbert-Schmidt operator
that satisfying
\[\lVert(W(g)\rVert_{HS}^2=\int_{K\times K}\left|\mathcal K_g(s,k)\right|^2dsdk
=\int_{K\times K}\left|g(s^{-1}k)\right|^2dsdk=\lVert g\rVert_2^2.\]
In other words, $W$ maps $L^1(K)$ onto  $S_2,$ the space of Hilbert-Schmidt
operators on $H(K, \mathbb C^{d_\sigma}).$

\smallskip

Next, we prove the Fourier inversion formula for the transform $W.$
\begin{proposition}\label{prop3}
If $g\in C^2(K),$ then the transform $W$ satisfies the inversion
formula $g(t)=\text{tr}(\pi(t)^*W(g)).$
\end{proposition}
\begin{proof}
Given that $g\in C^2(K),$
\begin{eqnarray*}
(\pi(t))^*W(g)&=&\int_{K} g(s)(\pi(t))^*\pi(s)ds=\int_{K} g(s)\pi(t^{-1})\pi(s)ds\\
&=&\int_{K} g^t(p)\pi(p)dp=W(g^t),
\end{eqnarray*}
where $g^t(p)=g(tp).$
That is, $W(g)$ is an integral operator with kernel $\mathcal K_g.$ Since the kernel
$\mathcal K_{g^t}$ satisfies $\mathcal K_{g^t}(s,k)=g^t(s^{-1}k),$ we obtain
$\mathcal K_{g^t}(s,s)=g(t)$ and hence
\begin{eqnarray*}
\text{tr}[\pi(t)^*W(g)]&=&\text{tr}(W(g^t))=\int_{K}\mathcal K_{g^t}(s,s)ds\\
&=&\int_{K}g(t)ds=g(t).
\end{eqnarray*}
\end{proof}

Further, by using  the Peter-Weyl  theorem, we prove that if $g\in L^1(K),$ then the
operator $W(g)$ has finite rank as long as $g$ is a trigonometric polynomial.

\begin{proposition}\label{prop2}
Let $g\in L^1(K).$ Then the operator $W(g)$ is of finite rank if and only if $g$
is a trigonometric polynomial on $K.$
\end{proposition}

\begin{proof}
Consider the function $h=g\ast g^\ast,$ where $g^\ast(t)=\overline{g(t^{-1})}.$
Now, we show that  $W(h)=W(g)^\ast W(g).$ For this, we have
\begin{eqnarray*}
W(h)&=&\int_{K} h(t)\pi(t)dt=\int_{K}(g\ast g^\ast)(t)\pi(t)dt\\
&=&\int_{K}\int_{K}g(s)g^\ast(ts^{-1})\pi(t)dtds\\
&=&\int_{K}g(s)\left(\int_{K}g^\ast(ts^{-1})\pi(t)dt\right)ds.
\end{eqnarray*}
By the change of variables $ts^{-1}=p$ in the inner integral, we get
\begin{eqnarray*}
W(h)&=&\int_{K}g(s)\left(\int_{K}g^\ast(p)\pi(ps)dp\right)ds\\
&=& W(g^\ast) W(g).
\end{eqnarray*}
Further, we require proving $W(g)^\ast=W(g^\ast).$ For $\varphi,\psi\in H(K,\mathbb C^{d_\sigma}),$ consider
\[\left\langle W(g^\ast)\varphi,\psi\right\rangle=\int_{K}\overline{g(t^{-1})}\left\langle\pi(t)\varphi,\psi\right\rangle dt=
\int_{K} \overline{g(s)}\left\langle\pi(s^{-1})\varphi,\psi\right\rangle ds.\]
Since $\pi$ is the left regular representation of $K,$ the operator  $\pi(s)$
will be unitary.
Hence
\[\left\langle W(g^\ast)\varphi,\psi\right\rangle=\int_{K} \left\langle\varphi,g(s)\pi(s)\psi\right\rangle ds=\left\langle\varphi,W(g)\psi\right\rangle=\left\langle W(g)^\ast\varphi,\psi\right\rangle.\]
This, in turn, implies that $W(h)=W(g)^\ast W(g)$ is a positive finite
rank operator. Thus, by the spectral theorem, there exists an orthonormal
set  $\{\varphi_{j}\in H(K,\mathbb C^{d_\sigma}):j=1,\ldots,m\}$ and
scalars $\lambda_j\geq0$ such that
\begin{equation}\label{exp42}
W(h)\varphi=\sum_{j=1}^{m}\lambda_{j}\left\langle \varphi,\varphi_{j} \right\rangle \varphi_{j},
\end{equation}
whenever $\varphi\in H(K,\mathbb C^{d_\sigma}).$ Let $\varphi_{j}=\left(\varphi_{j,1},\ldots,\varphi_{j,d_\sigma}\right).$
Then by (\ref{exp42}), it follows that $h\ast\varphi_{j,q}=\lambda_{j}\varphi_{j,q}.$
By taking Fourier coefficient of both the sides, we get
\[\widehat{\varphi^\delta}_{j,q}(\widehat{h^\delta}-\lambda_{j})=0,\]
where $\delta\in\hat K.$ Then $\widehat{h^\delta}=\lambda_j$ for finitely
many $\delta\in\hat K,$ otherwise, by the Riemann-Lebesgue lemma $\lambda_j=0.$
Hence ${\hat\varphi}_{j,q}(\delta)\neq0$ at most for finitely many
$\delta\in\hat K.$ Thus, by  the Peter-Weyl  theorem, we infer that $\varphi_{j,q}$
is a trigonometric polynomial. Since $h=g^\ast\ast g,$  it follows that
$\hat h(\delta)=|\hat g(\delta)|^2.$ Thus, from (\ref{exp42}), we conclude that
$g$ is a trigonometric polynomial.

\smallskip

Conversely, suppose $g$ is a trigonometric polynomial. Then without loss of
generality, we can assume that $g=\varphi_{ij}^\delta.$ Now, we can write
\[\varphi_{ij}^\delta(t^{-1}s)=\left\langle \delta(t^{-1}s)e_{j},e_i\right\rangle=
\left\langle\delta(s)e_{j},\delta(t)e_i\right\rangle.\]
Since $H_{\delta}$ is $\pi$-invariant,  it follows that
\[\varphi_{ij}^\delta(t^{-1}s)=\sum_{l=1}^{d_\delta}\varphi_{lj}^\delta(s)\overline{\varphi_{li}^\delta(t)}.\]
A straightforward calculation leads to
$W(g)\varphi(s)=d_\delta\left\langle\varphi,\varphi_{ij}^\delta\right\rangle\varphi_{ij}^\delta.$
Thus, we conclude that $W(g)$ is of finite rank.
\end{proof}

\begin{remark}\label{rk3}
In view of the Minkowski integral inequality, it can be easily seen that
$\|W(g)\|\leq\|g\|_1.$ Hence, the spectral radius of the operator $W(g)$
will satisfy the condition $\lambda[\sigma(W(g))]\leq\|g\|_1.$
\end{remark}

Next, by using Proposition \ref{prop2}, we prove that a radial function on the
motion group $G$ can be determined by its group Fourier transform at a
single point. However, for a sake of simplicity, we prove the result for $G=M(2).$
For proving this result, we need the following lemma.

\begin{lemma}\label{lemma1}
Let $f\in L^1(G)$ be such that $f(x,s)=f(|x|,s).$ If $\hat f(a_o, \sigma)$
is of finite rank for some $a_o\in\mathbb R_+,$ then
\[\int_0^\infty J(a_ot) F_t(s)tdt=\sum_{|n|\leq\alpha_o}\int_0^\infty J(a_ot)\hat F_t(n)Y_n(s)tdt,\]
where $F_t(s)=f(t,s).$
\end{lemma}
\begin{proof}
We know that for $\varphi\in L^2(K, \mathbb C^{d_\sigma})$ we have
\begin{eqnarray*}
(\hat f(a_o, \sigma)\varphi)(k) &=&\int_{\mathbb R^n}\int_{K} f(x,s)e^{-i\left\langle x,a_ok\right\rangle}\varphi(s^{-1}k)dxds\\
&=&\int_{K} \mathcal F_1f(a_ok,s)\varphi(s^{-1}k)ds.
\end{eqnarray*}
Since $f$ is radial in the first variable, then it follows that
\[(\hat f(a_o, \sigma)\varphi)(k)=\left(W\left(\mathcal F_1f(a_o, .\right)\right)\varphi)(k).\]
By the hypothesis, $\hat f(a_o, \sigma)$ is a finite rank operator,
$W\left(\mathcal F_1f(a_o, .\right))$ must be of finite rank. From Proposition \ref{prop2},
we conclude that $\mathcal F_1f(a_o, .)$ is a trigonometric polynomial. That is,
\begin{equation}\label{exp75}
\mathcal F_1f(a_o, s)=\sum_{|m|\leq\alpha_o} \hat G_{a_o}(m)\chi_m(s),
\end{equation}
where $G_{a_o}(s)=\mathcal F_1f(a_o, s).$
On the other hand, we have
\begin{eqnarray}\label{exp76}
\mathcal F_1f(a_o, s)&=&\int_0^\infty\int_{K}f(|t\omega|,s)e^{-iap\cdot t\omega}d\omega tdt\nonumber\\
&=&\int_0^\infty J_0(a_ot)f(t,s) tdt.
\end{eqnarray}
Now, we have
\begin{eqnarray*}
\hat {G_a}(m)&=&\int_{K}\mathcal F_1f(a_o,k)\chi_{-m}(k)dk\\
&=&\int_0^\infty J_0(a_ot)\left(\int_{K}f(t,k)\chi_{-m}(k)dk\right)tdt\\
&=&\int_0^\infty J_0(a_ot)\hat F_t(m)tdt,
\end{eqnarray*}
where $F_t(k)=f(t,k).$ Hence from (\ref{exp75}) we get
\begin{equation}\label{exp79}
\mathcal F_1f(a_o, s)=\sum_{|m|\leq\alpha_o}\int_0^\infty J_0(a_ot)\hat F_t(m)\chi_m(s)tdt.
\end{equation}
By comparing (\ref{exp76}) with (\ref{exp79}), we get the required identity.
\end{proof}

\begin{remark}\label{rk1}
Notice that by taking inverse Fourier transform in both the sides of (\ref{exp75}),
we can assume $f$ is trigonometric polynomial as long as $\hat f(a_o, \sigma)$
is a finite rank operator for some $a_o\in\mathbb R_+$ and $\sigma\in\hat M.$
\end{remark}

\begin{theorem}\label{th10}
Let $f\in L^1(G)$ be a radial function in the first variable which integrates
zero in the second variable. If $\hat f(a_o, \sigma)$ is a finite rank operator
for some $a_o\in\mathbb R_+$ and $\text{sign}(J_0f)\geq0,$ then $f=0.$
\end{theorem}
\begin{proof}
In view of Remark \ref{rk1}, from  Lemma \ref{lemma1} we infer that
\[\int_0^\infty\hat F_t(o)tdt=\int_0^\infty f(t,s)tdt.\]
This, in tern, implies that
\[\int_0^\infty\left(\int_{K}f(t,k)dk\right)tdt=\int_0^\infty f(t,s)tdt.\]
Since $f$ integrates zero on $K$ and $\text{sign}(J_0f)\geq0,$ we conclude
that $f=0.$
\end{proof}

\section{ Some results on the Heisenberg uniqueness pairs}\label{section5}
In this section, we explore the Heisenberg uniqueness pairs for the
Fourier transform on the motion group $G$ as well as
on the product group $G'=\mathbb R^n\times K,$ where $K$ is a compact
group. Further, we observed a one to one correspondence between the
class of HUP's on $\mathbb R^n$ and the class of HUP's on $G'.$

\bigskip

Let $\Gamma$ be a smooth surface (or a finite union of smooth surfaces)
in $\mathbb R^n$ and $\Gamma'=\Gamma\times K.$ Let $X(\Gamma')$ be the
space of all finite complex-valued Borel measures $\mu$ in the motion
group $G$ which is supported on $\Gamma'$ and absolutely continuous
with respect to the surface measure on $\Gamma'.$

\smallskip

We define the Fourier transform of $\mu$ on $G$ by

\begin{equation}\label{exp46}
(\hat \mu(a,\sigma)\varphi)(k)=\int_{\Gamma}\int_{K} f(x,s)e^{-i\left\langle x,ak\cdot e_n\right\rangle}\varphi(s^{-1}k)d\mu(x)ds,
\end{equation}
where $a\in\mathbb R^+$ and $\varphi\in H(K, \mathbb C^{d_\sigma}).$

\begin{theorem}
Let $\Gamma'=S^{n-1}\times K,$ where $S^{n-1}$ is the unit sphere in
$\mathbb R^n.$ Suppose $\mu\in X(\Gamma')$ is such that $\hat\mu(a_o,\sigma)=0$
for some $a_o\not\in J_{(n+2l-2)/2}^{-1}(0);\forall~l\in\mathbb Z$ and
$\sigma\in\hat M,$ then $\mu=0.$
\end{theorem}

\begin{proof}
Since $\mu$ is absolutely continuous with respect to the surface measure on
$\Gamma',$  by Radon-Nikodym theorem, there exists a function $f\in L^1(\Gamma')$
such that $d\mu=f dsdt.$ By hypothesis, we have
\[(\hat\mu(a_o,\sigma)\varphi)(k)=\int_{S^{n-1}}\int_{K} f(t,s)e^{-i\left\langle t,a_ok\cdot e_n\right\rangle}\varphi(s^{-1}k)dtds=0,\]
whenever $\varphi\in C(K, \mathbb C^{d_\sigma}).$  Now, by Fubini's theorem, we can write
\[\int_{K}\int_{S^{n-1}} f(t,s)e^{-i\left\langle t,a_ok\cdot e_n\right\rangle}dt\varphi(s^{-1}k)ds=
\int_{K}\mathcal F_1f(a_ok\cdot e_n,s)\varphi(s^{-1}k)ds=0.\]
Hence $\mathcal F_1f(a_ok\cdot e_n,s)=0$ for almost all $s,k\in K.$ Since $SO(n)$ can be identified
with $S^{n-1}$ via $k\rightarrow k\cdot e_n,$ it follows that $\mathcal F_1f(y,s)=0$
for almost all $y\in S^{n-1}_{a_o}(o)$ and $s\in K.$ Since we know from Theorem \ref{th14} that the
pair $\left(S^{n-1},S^{n-1}_{a_o}(o)\right)$ is a HUP as long as $J_{(n+2l-2)/2}(a_o)\neq0 $ for all
$l\in\mathbb Z_+,$ we conclude that $\mu=0.$
\end{proof}
\begin{remark}\label{rk41}
Let $(\Gamma,K)$ be a HUP in $\mathbb R^n$ and suppose $\mu\in X(\Gamma')$ is such that
$\hat \mu(a_o)=0$ for some $a_o\not\in J_{(n+2l-2)/2}$ and $\forall~l\in\mathbb Z_+,$
then $\mu=0.$
\end{remark}

\smallskip

The Haar measure on the product group $G'$ is given by $dg=dxdk,$ where $dx$ is
Lebesgue measure on $\mathbb R^n$ and $dk$ is normalized Haar measure on $K.$ Since
the unitary dual group of $G'$ can be parameterized by $\hat G'=\mathbb R^n\times\hat K,$
for each  $(y,\delta)\in\hat G,$ the map $(x,k)\mapsto e^{-2\pi i x\cdot y}\delta(k)$
is a unitary operator on the Hilbert space $\mathcal H_\delta.$ Hence, we can define
the Fourier transform of the function $f\in\L^1(G')$ by
\begin{equation}\label{exp502}
\hat f(y,\delta)=\int_{\mathbb R^n}\int_{K} f(x,k)e^{-2\pi ix\cdot y}\delta(k^{-1})dxdk.
\end{equation}

Let $\Gamma'= \Gamma\times K,$ where $\Gamma$ is a smooth surface (or a finite union
of smooth surfaces) in $\mathbb R^n.$ Let $X(\Gamma')$ be the space of all finite
complex-valued Borel measure $\mu$ in $G'$ which is supported on $\Gamma'$ and
absolutely continuous with respect to the surface measure on $\Gamma'.$ Then by
the Radon-Nikodym theorem, there exists a function $f\in L^1(\Gamma')$ such that
$d\mu=f d\nu dk,$ where $\nu$ is the surface measure on $\Gamma.$

\smallskip

Now, the Fourier transform of the measure $\mu$ can be defined by

\begin{eqnarray}\label{exp52}
\hat\mu(y,\delta) &=&\int_{\Gamma}\int_{K}e^{-2\pi ix\cdot y}\delta(k^{-1})d\mu(x,k)\nonumber\\
&=&\int_{\Gamma}\int_{K} f(x,k)e^{-2\pi ix\cdot y}\delta(k^{-1})d\nu(x)dk.
\end{eqnarray}
\begin{theorem}
The pair $(\Gamma,\Lambda)$ is a Heisenberg uniqueness pairs in $\mathbb R^n$ if and
only if $(\Gamma', \Lambda\times\hat K)$ is a Heisenberg uniqueness pairs in $G'.$
\end{theorem}
\begin{proof}
Suppose  $(\Gamma,\Lambda)$ is a Heisenberg uniqueness pair in $\mathbb R^n$ and
$\mu\in X(\Gamma').$ Then by Fubini's theorem, the map $x\mapsto f(x,k)$ belongs
to $L^1(\Gamma,d\nu)$ for almost all $k\in K.$ Hence for $(k, \sigma)\in K\times\hat K,$
we can define the projection $f_{k,\sigma}$ of $f$ by

\begin{eqnarray}\label{exp55}
f_{k,\sigma}(x)=\int_K f(x,kh^{-1})\chi_\sigma(h)dh,
\end{eqnarray}
where $\chi_\sigma=\text{tr }\sigma(~\cdot~),$ the character of the representation
$\sigma.$ Thus, the Euclidean Fourier transform of the projection $f_{k,\sigma}$
gives
\begin{eqnarray}\label{exp56}
\hat f_{k,\sigma}(y) &=& \int_{\Gamma}\int_K f(x,kh^{-1})e^{-2\pi ix\cdot y}\chi_\sigma(h)dh~d\nu(x)\nonumber\\
&=& \text{tr}\int_{\Gamma}\int_K f(x,kh^{-1})\sigma(h)dh~e^{-2\pi ix\cdot y}d\nu(x)\nonumber\\
&=& \text{tr}\int_{\Gamma}\int_K f(x,h)\sigma(h^{-1})dh~e^{-2\pi ix\cdot y}d\nu(x)\sigma(k)\nonumber\\
&=& \text{tr}\left(\hat\mu(y,\sigma)\sigma(k)\right).
\end{eqnarray}
Suppose $\hat \mu|_{\Lambda\times \hat K}=0.$ Since $(\Gamma,\Lambda)$ is a Heisenberg
uniqueness pair in $\mathbb R^n,$ from (\ref{exp56}), it follows that $f_{k,\sigma}=0.$
Hence by the uniqueness of the Fourier series
\[f(x,k)=\sum_{\sigma\in\hat K}d_\sigma f_{k,\sigma}(x)\]
we conclude that $f=0.$

\smallskip

Conversely, suppose $(\Gamma',\Lambda\times \hat K)$ is a Heisenberg
uniqueness pair in $G'.$ Then for $\mu\in X(\Gamma),$ there exists a
function $f\in L^1(\Gamma)$ such that $d\mu=fd\nu.$ Suppose $\hat\mu|_\Lambda=0.$
Then
\[\int_\Gamma f(x)e^{-2\pi i{x\cdot y}}d\nu(x)=0\]
for each $y\in\Lambda.$ This, in turn, implies
\begin{equation}\label{exp13}
\int_\Gamma\int_K f(x)e^{-2\pi i{y\cdot x}}\delta(k^{-1})dk~d\nu(x)=0.
\end{equation}
Now, if we write $d\rho=fd\nu dk,$ then $\rho\in X(\Gamma').$ Since
$(\Gamma',\Lambda\times \hat K)$ is a Heisenberg uniqueness pair,
by (\ref{exp13}), it follows that $\rho=0.$ Thus, using the fact that
group compact group $K$ is unimodular, we conclude that the measure $\mu=0.$
\end{proof}

\bigskip

\noindent{\bf Concluding remarks:}\\

We have shown that if the Fourier transform of a function in $L^1\cap L^2(M(n))$
lands into the space of finite rank operators, then the function has to
vanish. However, it would be an interesting question to consider the case when
the spectrum of the group Fourier transform of a function $L^1(M(n))$ is supported
on a thin uncountable set. We leave this question open for the time being.

\bigskip

\noindent{\bf Acknowledgements:}\\

The authors wish to thank E. K. Narayanan for several fruitful discussions
and remarks. The authors would like to gratefully acknowledge the support
provided by IIT Guwahati, Government of India.

\bigskip


\end{document}